\documentclass[12pt]{amsart}
\usepackage{amscd,amsmath,amsthm,amssymb}
\usepackage{color}
\usepackage{amsfonts,amssymb,amscd,amsmath,enumerate,verbatim,enumitem}
\usepackage{lineno}
 \def\NZQ{\mathbb}               
 \def\NN{{\NZQ N}}

 \def\frk{\mathfrak}               

 \def\mm{{\frk m}}

 \def\opn#1#2{\def#1{\operatorname{#2}}} 
 \opn\chara{char} \opn\length{\ell} \opn\pd{pd} \opn\rk{rk}
 \opn\projdim{proj\,dim} \opn\injdim{inj\,dim} \opn\rank{rank}
 \opn\depth{depth} \opn\grade{grade} \opn\height{height}
 \opn\embdim{emb\,dim} \opn\codim{codim}
 
 \opn\Tr{Tr} \opn\bigrank{big\,rank}
 \opn\superheight{superheight}\opn\lcm{lcm}
 \opn\trdeg{tr\,deg}
 \opn\reg{reg} \opn\lreg{lreg} \opn\ini{in} \opn\lpd{lpd}
 \opn\size{size} \opn\sdepth{sdepth}
 \opn\link{link}\opn\fdepth{fdepth}\opn\lex{lex}
 \opn\div{div} \opn\Div{Div} \opn\cl{cl} \opn\Cl{Cl}
 \opn\Spec{Spec} \opn\Supp{Supp} \opn\supp{supp} \opn\Sing{Sing}
 \opn\Ass{Ass} \opn\Min{Min}\opn\Mon{Mon}
 \opn\Ann{Ann} \opn\Rad{Rad} \opn\Soc{Soc}
 \opn\Im{Im} \opn\Ker{Ker} \opn\Coker{Coker} \opn\Am{Am}
 \opn\Hom{Hom} \opn\Tor{Tor} \opn\Ext{Ext} \opn\End{End}
 \opn\Aut{Aut} \opn\id{id}
 
 \opn\nat{nat}
 \opn\pff{pf}
 \opn\Pf{Pf} \opn\GL{GL} \opn\SL{SL} \opn\mod{mod} \opn\ord{ord}
 \opn\Gin{Gin} \opn\Hilb{Hilb}\opn\sort{sort}
 \opn\aff{aff} \opn
 \con{conv} \opn\relint{relint} \opn\st{st}
 \opn\lk{lk} \opn\cn{cn} \opn\core{core} \opn\vol{vol}  \opn\inp{inp} \opn\nilpot{nilpot}
 \opn\link{link} \opn\star{star}\opn\lex{lex}\opn\set{set}
 \opn\width{wd}
 \opn\gr{gr}
 
 \def\pot#1#2{#1[\kern-0.28ex[#2]\kern-0.28ex]}

 \opn\dirlim{\underrightarrow{\lim}}
 \opn\inivlim{\underleftarrow{\lim}}
 \let\to=\rightarrow
 
 \def\Implies{\ifmmode\Longrightarrow \else
         \unskip${}\Longrightarrow{}$\ignorespaces\fi}
 \def\implies{\ifmmode\Rightarrow \else
         \unskip${}\Rightarrow{}$\ignorespaces\fi}
 \def\iff{\ifmmode\Longleftrightarrow \else
         \unskip${}\Longleftrightarrow{}$\ignorespaces\fi}

 \let\:=\colon
 \newtheorem{Theorem}{Theorem}[section]
 \newtheorem{Lemma}[Theorem]{Lemma}
 \newtheorem{Corollary}[Theorem]{Corollary}
 \newtheorem{Proposition}[Theorem]{Proposition}
 \newtheorem{Remark}[Theorem]{Remark}
 
 \newtheorem{Example}[Theorem]{Example}

 \let\epsilon\varepsilon
 \let\kappa=\varkappa
 \def\qed{\ifhmode\textqed\fi
       \ifmmode\ifinner\quad\qedsymbol\else\dispqed\fi\fi}
 \def\textqed{\unskip\nobreak\penalty50
        \hskip2em\hbox{}\nobreak\hfil\qedsymbol
        \parfillskip=0pt \finalhyphendemerits=0}
 \def\dispqed{\rlap{\qquad\qedsymbol}}
 
 \opn\dis{dis}
 \def\pnt{{\raise0.5mm\hbox{\large\bf.}}}
 
 \opn\Lex{Lex}

\begin{document}

\title {On the Cohen-Macaulay property for quadratic tangent cones}

 \author {Dumitru I.\ Stamate}

\address{Dumitru I. Stamate, Faculty of Mathematics and Computer Science, University of Bucharest, Str. Academiei 14, Bucharest, Romania, and  \newline  \indent
Simion Stoilow Institute of Mathematics of the Romanian Academy, Research group
of the project PN-II-RU-PD-2012-3-0656, P.O.Box 1-764, Bucharest 014700, Romania}
\email{dumitru.stamate@fmi.unibuc.ro}

\dedicatory{ }

\begin{abstract}
Let $H$ be an $n$-generated numerical semigroup such that its tangent cone $\gr_\mm K[H]$ is defined by quadratic relations.
We show that  if $n<5$ then $\gr_\mm K[H]$ is Cohen-Macaulay, and for $n=5$ we explicitly describe the   semigroups $H$  such that   $\gr_\mm K[H]$ is not Cohen-Macaulay.
As an application we show that if the field $K$ is algebraically closed  and of characteristic different from two, and $n\leq 5$ then $\gr_\mm K[H]$ is Koszul 
if and only if (possibly after a change of coordinates) its defining ideal has a quadratic Gr\"obner basis. 
\end{abstract}

\thanks{}
 

\maketitle


\section*{Introduction}

A numerical semigroup $H$ is a subset of $\NN$ containing $0$ and which is closed under addition such that the gcd of all elements in $H$ is $1$, or
equivalently, such that $|\NN \setminus H| < \infty$. We denote  $Gen(H)$ its unique minimal generating set. 
The embedding dimension of $H$ is defined as $\embdim(H)= |Gen(H)|$ 
and the multiplicity of $H$ is $e(H)= \min Gen(H)$.

Let $K$ be an infinite field.
The additive relations among the generators of $H$ are captured by the defining ideal $I_H$ of the semigroup ring $K[H]=K[t^h: h\in H]\subset K[t]$.  
Namely, if $Gen(H)=\{a_1, \dots, a_n\}$ and we let $S=K[x_1, \dots, x_n]$, then $I_H=\Ker \phi$, 
where $\phi:S\to K[H]$ is the $K$-algebra map with $\phi(x_i)=t^{a_i}$, for $1\leq i\leq n$.

Another algebra that is associated to $H$ is its tangent cone  $$\gr_\mm K[H]=\oplus_{i\geq 0} \mm^i/\mm^{i+1},$$
which is the  associated graded ring  of $K[H]$ with respect to the maximal ideal $\mm=(t^h: h\in H)$. 
The defining ideal of $\gr_\mm K[H]$ is $I_H^*$, the ideal of initial forms in $I_H$, see \cite[\S 15.10.3]{Eis} and \cite[\S 3.4]{EH}.

It is a classical topic to study algebraic properties of $\gr_\mm K[H]$ like
being Cohen-Macaulay or complete intersection (CI for short) in terms of the arithmetic of $H$, see \cite{He-reg}, \cite{RV}, \cite{BarF}, \cite{dAMS}.

Algebras defined by quadratic relations occur naturally in algebraic geometry from varieties cut out by quadrics 
and they have been the initial framework for formulating several strong conjectures, 
e.g. what is currently known as the Eisenbud-Green-Harris conjecture introduced in \cite[Section 4]{EGH-highCastelnuovo}.

In recent work (\cite{HeS-quad}),  J.~Herzog and the author gave effective bounds for the multiplicity of a 
numerical semigroup $H$  such that $\gr_\mm K[H]$ is defined by quadrics.
The motivation for the current  paper came from the puzzling observation that all such numerical semigroups that 
we had obtained by blind computer search  have the property that $\gr_\mm K[H]$ is Cohen-Macaulay. 

Koszul algebras are an important class of quadratic algebras. 
Recall that a graded $K$-algebra $R= \oplus_{i\geq 0} R_i$ is called Koszul if $K$ has a graded $R$-linear resolution. 
$R$ is called $G$-quadratic if there exists a graded isomorphism  
$R\cong K[x_1,\dots, x_n]/I$ where $I$ has a quadratic Gr\"obner basis with respect to some term order.
It is well known that if $R$ is $G$-quadratic, then it is Koszul.
We refer to \cite{CdNR} and \cite{EH} for proofs and related results on Koszul algebras.

For brevity, we say that a numerical semigroup $H$ is quadratic, Koszul, or $G$-quadratic, if $\gr_\mm K[H]$ has the respective property.
Note that the quadratic property for $\gr_\mm K[H]$ depends on $H$ alone (see  \cite[Theorem 6.8]{RS}), however the other two might depend on the field $K$. 
It will be clear from the context, mainly in Section \ref{sec:g-quadratic}, which are our extra assumptions on $K$.

Let $H$ be a quadratic numerical semigroup.
Using a criterion obtained independently  by J.~Herzog (\cite{He-reg}) and A.~Garcia (\cite{Garcia}), 
and also our  results from \cite{HeS-quad}, we show in Proposition \ref{prop:cm-small-n} that
if $\embdim(H)<5$, then $\gr_\mm K[H]$ is Cohen-Macaulay. 
It requires a bit more work to prove in Theorem \ref{thm:5-quad-non-cm} that if $\embdim(H)=5$,
 then  $\gr_\mm K[H]$ is not Cohen-Macaulay precisely when $H$ is  generated as
\begin{eqnarray*}
 \langle 8, 4u', 4u+2u', 4u''+2u+ u', 6u+7u'+4u''-8 \rangle, \text{ or} \\
 \langle 8, 4u', 4u+2u', 4u''+2u+3u', 6u+9u'+4u''-8 \rangle,
\end{eqnarray*} with $u, u', u''$ positive integers and $u'>1$ is odd.
Extending these exemples, in Proposition \ref{prop:construction}  for any $n> 5$ we construct infinitely many $G$-quadratic numerical semigroups $H$ with
$\embdim(H)=n$ and $\gr_\mm K[H]$ not Cohen-Macaulay. 
It would be interesting to figure out if such constructions cover all the cases when $\gr_\mm K[H]$ is quadratic and not Cohen-Macaulay.

In the terminology introduced by Rossi and Valla in \cite{RossiValla}, for $H$ in the above families the ideals $I_H$ provide  first examples of $2$-isomultiple ideals such that
$\gr_\mm S/I_H$ is not Cohen-Macaulay.  Their existence was also questioned in \cite[Remark 2.3]{RossiValla}.
 
For a standard graded $K$-algebra $R=\oplus_{i\geq 0} R_i$ its $h$-polynomial $h(z)=\sum_{i\geq 0 }h_i z^i$ is the numerator of the Hilbert series 
$H_R(z)=\sum_{i\geq 0} \dim_K R_i z^i$ when we write
$H_R(z)=h(z)/(1-z)^d$ with $h(1)\neq 0$. The $h$-vector of $R$ is the vector of coefficients $(h_0, h_1, \dots)$ of the $h$-polynomial.
Also, the (Hilbert-Samuel) multiplicity of $R$ is defined as $e(R)= h(1)$. It is known that for a numerical semigroup $H$ its multiplicity equals the multiplicity
 of the tangent cone $\gr_\mm K[H]$.

By  work of Backelin, Conca and others, small values of $h_2$ imply the Koszul or the $G$-quadratic property of $R$, 
see  \cite{Backelin-thesis}, \cite{Conca-quadrics}, \cite{Conca-quadrics-codim3}, \cite{DAli} and Lemma \ref{lemma:tools}.
If $R$ is Cohen-Macaulay and the field $K$ is infinite, we can mod out by a regular sequence of linear forms and the $h$-vector and  the multiplicity  are preserved.
In case $H$ is a numerical semigroup and  $R=\gr_\mm K[H]$ is Cohen-Macaulay, we may use  $t^{e(H)}$ as a regular element.

As an application, in Section \ref{sec:g-quadratic} we show that if $\embdim(H) <5$, then $H$ is quadratic if and only if it is $G$-quadratic.
The first examples of quadratic and non-Koszul semigroups occur in embedding dimension $5$ having  multiplicity $9$, e.g. $H=\langle 9, 17, 20, 23, 25 \rangle$.

In a similar way, in \cite{Roos-Sturmfels} Roos and Sturmfels considered the Koszul property for quadratic projective monomial curves. Namely, given the relatively prime integers
$0=a_1<a_1<\dots <a_n$, let $R=K[t_1^{a_i}t_2^{a_n-a_i}:1\leq i \leq n] \subset K[t_1, t_2]$. 
According to Table 1 in \cite{Roos-Sturmfels} obtained by a computer search,  the first time when $R$ is quadratic and not Koszul is  for $n=6$, and for $n=8$ occurs 
the first example where $R$ is Koszul and  the associated toric ideal has no quadratic Gr\"obner basis.

Under the assumption that the field $K$ is algebraically closed and of characteristic $\neq 2$, we show in Theorem \ref{thm:5semi-koszul}
that if $\embdim(H)=5$, then $H$ is Koszul if and only if it is $G$-quadratic. 
The proof works on the possible $h$-vectors of $\gr_\mm K[H]$ when $H$ is quadratic, employing  a result of Eisenbud, Green, and Harris in \cite{EGH-highCastelnuovo}.
The assumptions on the field $K$ are due to   Conca's results on the $G$-quadratic property for quadratic algebras with $h_2 \leq 3$, 
see \cite{Conca-quadrics} and \cite{Conca-quadrics-codim3}. Screening the possible ideals $J=I_H^* \mod x_1$ we found only two possible
situations without a quadratic Gr\"obner basis, described in Remark \ref{rem:h3-nonflag}. However, experimentally we found no quadratic semigroup $H$ producing such ideals.

We summarize our findings in Table \ref{table:multiplicities} in Section \ref{sec:g-quadratic} 
where we give a maximal list of $12$ possible $h$-vectors of quadratic $5$-generated numerical semigroups. 
Note that experimentally we could not obtain the $h$-vector $(1,4,5)$. Nevertheless, we can conclude that if $H$ is quadratic and $\embdim(H) \leq 5$, the
Hilbert function of $\gr_\mm K[H]$ is non-decreasing. This   topic  has   been recently considered by D'Anna, Di Marca and Micale in \cite{dAdMM} and by
Oneto, Strazzanti and Tamone in \cite{OST}.


\section{The Cohen-Macaulay condition}

In this section we study the Cohen-Macaulay property for the tangent cone of a quadratic numerical semigroup.

For further reference we first recall from our joint work with J.~Herzog \cite{HeS-quad}   some restrictions that we found on the multiplicity  of a quadratic numerical semigroup. 

\begin{Theorem}
\label{thm:bounds}
(\cite[1.1, 1.9, 1.12]{HeS-quad})
Let $H$ be a quadratic numerical semigroup minimally generated by $n>1$ elements and $K[H]$ its semigroup ring. Then
\begin{enumerate}
\item [(i)] $n \leq e(H) \leq 2^{n-1}$;
\item [(ii)] $e(H)=n$ \iff $I_H^*$ has a linear resolution;
\item [(iii)] $e(H)=2^{n-1}$ \iff $I_H^*$ is a CI ideal \iff $I_H$ is a CI ideal;
\item [(iv)] if $\gr_\mm K[H]$ is Cohen-Macaulay and $e(H) <2^{n-1}$, then $e(H) \leq 2^{n-1}-2^{n-3}$.
\end{enumerate}
Moreover, if we are in any of the situations from $(ii)$, $(iii)$ or  if $\gr_\mm K[H]$ is Cohen-Macaulay and  $e(H) = 2^{n-1}-2^{n-3}$ then $H$ is $G$-quadratic, hence Koszul.
\end{Theorem}

\begin{Remark}{\em
\label{rem:minimal}
With notation as above, if $e(H)=n$, then   $\gr_\mm K[H]$ has minimal multiplicity and by Sally's \cite[Theorem 2]{Sally-cm} we get that $\gr_\mm K[H]$ is Cohen-Macaulay.
We refer to the proof of Proposition 1.3 in \cite{HeS-quad} for related properties.
}
\end{Remark}

The following  arithmetic result appeared in \cite{HeS-quad}.

\begin{Lemma}(\cite[Lemma 1.6]{HeS-quad}) 
\label{lemma:quad-semi}
Let $H$ be a numerical semigroup minimally generated by $a_1<a_2<\dots <a_n$ with  $n>1$.
If $H$ is quadratic, then 
\begin{enumerate}
\item[(i)] there exist $ k, \ell \geq 2$ such that $a_1| a_k +a_\ell$.
\item [(ii)] $2 a_i \in \langle a_1, \dots a_{i-1}, a_{i+1}, \dots a_n\rangle$, for all $2\leq i \leq n$.
\end{enumerate}
\end{Lemma}

For the numerical semigroup  $H$  minimally generated by $a_1<\dots <a_n$ we denote 
$$c_i= \min \{k>0: ka_i \in \langle Gen(H)\setminus\{a_i \}\rangle \}, \text{ for } i=1, \dots, n.
$$
With this notation one has the following characterization proved independently by Herzog (\cite{He-reg}) and Garcia (\cite{Garcia}).

\begin{Proposition}(Herzog \cite[pp.189-190]{He-reg}, Garcia \cite[Theorem 24]{Garcia})
\label{prop:cm-gr}
The tangent cone $\gr_\mm K[H]$ is  Cohen-Macaulay  if and only if for all  integers  $0\leq \nu_i < c_i$ and $2\leq i \leq n$ such that
$$
\sum_{i=2}^n \nu_i a_i \in a_1+H,
$$
there exist integers $\mu_1>0, \mu_2 \geq 0, \dots, \mu_n\geq 0$ such that
$$
\sum_{i=2}^n \nu_i a_i= \sum_{i=1}^n \mu_i a_i \text{\quad and \quad} \sum_{i=2}^n \nu_i \leq \sum_{i=1}^n \mu_i.
$$
\end{Proposition}

If $H$ is quadratic, by Lemma \ref{lemma:quad-semi}(ii) we have $c_i=2$ for $i=2, \dots, n$.
This observation, together with Proposition \ref{prop:cm-gr}, gives the next result.

\begin{Proposition}
\label{prop:cm-small-n}
If $H$ is any quadratic semigroup with $\embdim(H) <5$ then $\gr_\mm K[H]$ is Cohen-Macaulay. 
\end{Proposition}

\begin{proof}
If  {$\embdim(H)=2$} then $H=\langle 2, \ell \rangle$ with $\ell >1$ odd. Hence $\gr_\mm K[H] \cong K[x_1,x_2]/(x_2^2)$ which is Cohen-Macaulay.

If  {$\embdim(H)=3$},  
by Theorem \ref{thm:bounds}(i) we have $3\leq e(H) \leq 4$, and by (ii) and (iii) in loc.cit.  $\gr_\mm K[H]$ is Cohen-Macaulay.

If  {$\embdim(H)=4$} let 
$$
\nu_2 a_2 +\nu_3 a_3+ \nu_4 a_4=\mu_1 a_1+ \mu_2 a_2 +\mu_3 a_3+ \mu_4 a_4
$$  
for some $\mu_1 >0, \mu_2, \mu_3 , \mu_4 \geq 0$ and $\nu_2, \nu_3, \nu_4 \in\{0,1\}$.
It is enough to consider the case when not both of $\nu_2$ and $\mu_2$,  of $\nu_3$ and $\mu_3$, and of $\nu_4$ and $\mu_4$ are positive.
Note that since $\embdim(H)=4$ at least two of the $\nu_i$'s need to be positive.

If $\nu_2=\nu_3=1$ and $\nu_4=0$ then in the equation $a_2+a_3=\mu_1 a_1+\mu_4 a_4$ we have either $\mu_4>0$, 
hence $\mu_1+\mu_4 \geq 2=\nu_2+\nu_3$, or $\mu_4=0$, hence $a_2+a_3=\mu_1 a_1$ with $\mu_1 >2$.
The cases $\nu_3=\nu_4=1, \nu_2=0$ and $\nu_2=\nu_4=1, \nu_3=0$ are treated similarly.

If $\nu_2=\nu_3=\nu_4=1$ then in the equation $a_2+a_3+a_4=\mu_1 a_1$ we have  $\mu_1>3$. 

By Proposition \ref{prop:cm-gr} it follows that $\gr_\mm K[H]$ is Cohen-Macaulay.
\end{proof}

\begin{Example}
\label{ex:quad-non-cm}
{\em
The statement of Proposition \ref{prop:cm-small-n} is no longer true when $\embdim(H)$ is at least $5$. We can check with Singular (\cite{Sing}) that for 
$H= \langle 8,12,13,18,35 \rangle$ the ideal $I_H^*$ has a quadratic Gr\"obner basis with respect to revlex, however $\gr_\mm K[H]$ is not Cohen-Macaulay.
 
Indeed, the toric ideal $I_H$ is minimally generated by
\begin{multline*}
I_H= (x_3^2-x_1 x_4, \ x_2 x_4^2-x_3 x_5,\ x_2 x_3 x_4-x_1 x_5, \ x_2^3-x_4^2,  \\ 
	x_1^3-x_2^2, \ x_1^2 x_3 x_4-x_2 x_5,\  x_1^2 x_4^3-x_5^2, \ x_1^2 x_2^2 x_3-x_4 x_5).
\end{multline*}

A revlex Gr\"obner basis for $I_H^*$ is given by
$$
I_H^*=(x_5^2, x_4 x_5, x_3 x_5, x_2 x_5, x_1x_5, x_4^2, x_3^2-x_1x_4,x_2^2),
$$
and $(I_H^*:x_5)= (x_1, x_2, x_3, x_4, x_5)$, hence $\depth_\mm K[H]=0$.
}
\end{Example}

This is not an isolated example. For any given embedding dimension $n>4$ we   construct infinitely many $G$-quadratic  
numerical semigroups whose tangent cone is not Cohen-Macaulay.  But first we recall a useful construction.

Let $L$ be a numerical semigroup,  $\ell$ an odd integer in $L$ and $H=\langle 2L, \ell \rangle$. By \cite[Definition 2.2]{HeS-quad}, the semigroup  $H$ is called a
quadratic gluing of $L$. It is proved in \cite[Corollary 2.7]{HeS-quad} that $L$ and $H$ are quadratic, Koszul, respectively G-quadratic, at the same time. 
It is also known by Delorme's work \cite{Delorme}  that if $L$ is a complete intersection (CI), then so is $H$.
We refer to Section 2 in \cite{HeS-quad} for more details about the CI property for quadratic numerical semigroups.

\begin{Proposition}
\label{prop:construction}
Given $n \geq 3$ and the positive integers $u_i$, $i=1, \dots, n+1$, where $u_1\geq 3 $ is odd, let 
\begin{eqnarray*}
a_1 &=& 2^{n}, \\
a_2 &=& 2^{n-1} u_1, \\
a_3 &=& 2^{n-1} u_2+ 2^{n-2} u_1,\\
\dots \\
a_{n+1} &=& 2^{n-1} u_n+ 2^{n-2} u_{n-1}+\dots + u_1,\\
a_{n+2} &=& a_2+\dots+ a_{n+1} -a_1.
\end{eqnarray*}
The semigroup  $H=\langle a_1, \dots, a_{n+2} \rangle$ is a G-quadratic  numerical semigroup of embedding dimension $n+2$ and $\gr_\mm K[H]$ is not Cohen-Macaulay. 
\end{Proposition}

Letting $n=3$, $u_1=u_2=3$ and $u_3=1$ in the construction  above,  we obtain the quadratic semigroup  $H= \langle 8, 12, 18, 13, 35 \rangle$ 
from Example \ref{ex:quad-non-cm}. Note that in Proposition \ref{prop:construction} the listed generators $a_i$ are not necessarily in increasing order, however 
we always have $a_1<a_i<a_{n+2}$ for $2\leq i \leq n+1$.

\begin{proof}
Denote $H_i=\langle a_1,\dots, a_i \rangle$ for $1\leq i \leq n+2$. It is easy to see that 
$$
\gcd(a_1,\dots, a_i)=2^{n-i+1} \text{ \quad for }i=1,\dots, n+1,
$$ 
hence $H_{n+1}$ and $H$ are  numerical semigroups.
From the defining relations we infer that 
\begin{equation}
\label{eq:recurence}
2a_{i+1}= 2^n u_i+ a_i \text{ \quad for } 2\leq i \leq n,
\end{equation}
hence the (not necessarily numerical) semigroups $H_1, \dots, H_{n+1}$ are obtained by quadratic gluings, 
are CI and $G$-quadratic.
Also, from the equations \eqref{eq:recurence} we see that 
$$
I_{H_i}= (x_2^2-x_1^{u_1})+(x_{j+1}^2-x_1^{u_j}x_j: 2\leq j <i) \quad \text{ for } 2\leq i \leq n+1.
$$ 
Next we compute $I_H$ and $I_H^*$. We note that
$$
a_{n+2}+ a_1= a_2+\dots + a_{n+1}.
$$
For $2\leq i \leq n+1$  using \eqref{eq:recurence} repeatedly we get  
\begin{eqnarray*}
a_{n+2}+ a_i &=& a_2+\dots + a_{i-1} + 2a_i + \sum_{j=i+1}^{n+1}a_j -2^n \\
	           &=& a_2+\dots + a_{i-1} +( a_{i-1} + 2^n u_{i-1} )+ \sum_{j=i+1}^{n+1}a_j  -2^n\\
						 &=& a_2+\dots + a_{i-2}+ (a_{i-2}+ 2^n u_{i-2}) + 2^n u_{i-1} + \sum_{j=i+1}^{n+1}a_j-2^n\\ 
						&\dots& \\
						 &=& 2^n(u_{i-1}+\dots +u_1 -1) + \sum_{j=i+1}^{n+1}a_j \\
						 &=& (u_{i-1}+\dots + u_1 -1) a_1 + \sum_{j=i+1}^{n+1}a_j.
\end{eqnarray*}
Arguing similarly we obtain
$$
2 a_{n+2}=  (u_1+\dots + u_n -2)a_1 +\sum_{i=2}^{n}a_i.
$$
Each of these relations produces a binomial in $I_H$:
\begin{eqnarray*}
f_1     &=& x_1 x_{n+2}- \prod_{j=2}^{n+1}x_j, \\ 
f_i     &=& x_i x_{n+2}- x_1^{(\sum_{j=1}^{i-1}u_j) -1} \prod_{j=i+1}^{n+1} x_j, \text{ for } 2\leq i \leq n+1, \text{ and} \\
f_{n+2} &=& x_{n+2}^2- x_1^{(\sum_{j=1}^n u_j)-2} \prod_{j=2}^n x_j.
\end{eqnarray*}
By inspecting these relations we remark that we can always choose a  generating set for $I_H$ consisting of binomials
such that in each monomial in the support, different from $x_{n+2}^2$, the variable  $x_{n+2}$ has degree at most one.
Therefore 
$$
I_H= (I_{H_{n+1}}, f_1, \dots, f_{n+2}).
$$

Since $n\geq 3$ and $u_1 \geq 3$ it is easy to see that $f_i^*=x_i x_{n+2}$ for $i=1,\dots, n+2$. Arguing as above we derive
$$
I_H^*=I_{H_{n+1}}^*+ x_{n+2}(x_1,\dots, x_{n+2}).
$$
This gives $I_H^*: x_{n+2}=(x_1,\dots, x_{n+2})$ and $\depth \gr_\mm K[H]=0$, hence  $ R=\gr_\mm K[H]$ is not Cohen-Macaulay.

Since $x_{n+2} R_1=0$, by Conca's \cite[Lemma 4.(1)]{Conca-quadrics} we have that $ R$ is $G$-quadratic 
if and only if 
$$ R/(t^{a_{n+2}})\cong K[x_1,\dots, x_{n+2}]/(I^*_H, x_{n+2})\cong K[x_1, \dots, x_{n+1}]/I^*_{H_{n+1}} \cong \gr_\mm K[H_{n+1}]$$
is $G$-quadratic, which is true since $H_{n+1}$ is a quadratic CI, see Theorem \ref{thm:bounds}.
\end{proof}

\medskip

Our next goal is to identify the quadratic numerical semigroups $H$ of embedding dimension $5$  and $\gr_\mm K[H]$ not Cohen-Macaulay.
\begin{Theorem}
\label{thm:5-quad-non-cm} Let $H$ be a quadratic numerical semigroup with $\embdim(H)=5$.
Then $\gr_\mm K[H]$ is not Cohen-Macaulay if and only if $H$ is of any of the following forms:
\begin{enumerate}
\item[(i)] $H=\langle 8, 4u', 4u+2u', 4u''+2u+ u', 6u+7u'+4u''-8 \rangle$ with $u, u', u''$ positive integers and $u'>1$ is odd, or
\item[(ii)] $H=\langle 8, 4u', 4u+2u', 4u''+2u+3u', 6u+9u'+4u''-8 \rangle$ with $u, u', u''$ positive integers and $u'>1$ is odd.
\end{enumerate}
Whenever $H$ is in any of these two families, it is also $G$-quadratic.
\end{Theorem}

\begin{proof}
We first assume $H$ is in any of the specified families and we show that $\gr_\mm K[H]$ is not Cohen-Macaulay.
We label $a_1,\dots, a_5$ the  generators  of $H$ in the given ordering.

For $(i)$ we are in the situation described in Proposition \ref{prop:construction} for $n=3$, $u_1= u'$, $u_2=u$ and $u_3=u''$, 
hence the conclusion follows.

For $(ii)$ we note that the semigroup $L=\langle a_1, a_2, a_3, a_4\rangle = \langle 2 \langle 4, 2u', 2u+u'\rangle, a_4 \rangle$ 
is obtained by a quadratic gluing since $a_4=u''\cdot 4+(2u+u')+ (2u')$ is odd. 
By \cite[Proposition 3.6]{HeS-quad} the semigroup $\langle 4, 2u', 2u+u' \rangle$ is a quadratic complete intersection, 
and the same holds for $L$ by Delorme's \cite[Proposition 9]{Delorme}  and by
\cite[Corollary 2.7]{HeS-quad}. 

It is straightforward to check that 
\begin{eqnarray*}
a_5+a_1 &=& a_2+a_3+a_4,\\
a_5+ a_2 &=& (u'-1)a_1+ a_3+a_4,\\
a_5+ a_3 &=& (u+u'-1)a_1+ a_4,\\
a_5+ a_4 &=& (u+u''-1) a_1 + 3 a_2,\\
2 a_5 &=& (u+2u'+u''-2) a_1 + a_3,
\end{eqnarray*}
hence $I^*_H:x_5=(x_1, \dots, x_5)$ and $\gr_\mm K[H]$ is not Cohen-Macaulay.
 Arguing as in the proof of Proposition \ref{prop:construction} we get that $H$ is $G$-quadratic.

The direct implication is  proved separately in Section \ref{sec:long}.
\end{proof}


\section{Koszul and $G$-quadraticity}
\label{sec:g-quadratic}

As an application of Theorem \ref{thm:5-quad-non-cm}, under some   restrictions on the field $K$, we prove that if $H$ is a   numerical semigroup and $\embdim(H) \leq 5$, then
$H$ is Koszul if and only if it is $G$-quadratic. 
We wonder if this statement holds for arbitrary embedding dimension.

Let $R=\oplus_{i\geq 0} R_i$  be a  standard graded $K$-algebra.
A {\em Koszul filtration}  for $R$  is a family $\mathcal{F}$ of ideals of $R$ generated by linear forms such that $0$ and 
the maximal homogeneous ideal of $R$ belong to $\mathcal{F}$ and for every $I\in \mathcal{F}$ different from $0$, there exists $J \in \mathcal{F}$ such that
$J\subset I$, $I/J$ is cyclic and $J:I \in \mathcal{F}$. A Koszul filtration that is totally ordered with respect to inclusion is called a {\em Gr\"obner flag}.
It is known that if $R$ has a Koszul filtration, then it is Koszul. Also, by \cite[Theorem 2.4]{CRV}, if $R$ has a Gr\"obner flag, then $R$ is $G$-quadratic.
We refer to the original papers  \cite{CTV}, \cite{CRV} and to the recent survey \cite{CdNR} of Conca et al.  for more properties.

For easier reference we group in the following lemma some known results about lifting Koszul-like properties modulo a linear form. 

\begin{Lemma}
\label{lemma:lifting}
Let $R$ be a standard graded $K$-algebra and $x$ a linear form that is regular on $R$. If $R/(x)$ has  property $(\mathcal{P})$, then so does $R$, where
$\mathcal{P}$ stands for Koszul, $G$-quadratic, admits a Koszul filtration, or a Gr\"obner flag.  
\end{Lemma}

\begin{proof} The statements for Koszul and $G$--quadraticity are due to Backelin and Fr\"oberg in \cite[Lemma 2]{BF-poincare}, respectively to Conca in 
\cite[Lemma 4.(2)]{Conca-quadrics}. 

That any Gr\"obner flag may be lifted from $R/(x)$ to $R$ is proved in \cite[Lemma 2.11.(a)]{CRV}. Using the same idea one can
produce a  Koszul filtration for $R$ from a Koszul filtration of  $R/(x)$.
\end{proof}

For $R$ as above its $h$-polyonomial is defined as the numerator $h(t)$ in the writing of the Hilbert series
$H_R(t)=\Sigma_{i\geq 0} \dim_K R_i t^i= h(t)/(1-t)^d$ with $h(1)\neq 0$. If $h(t)=\Sigma_{i\geq 0}h_i t^i$,  the sequence of coefficients $(h_0, h_1, \dots)$ is called  the $h$-vector of $R$.

It is clear that if $x$ is a linear form which is regular on $R$, the $h$-polynomial and the $h$-vector of $R$ and $R/(x)$ are the same.
In this context, the Cohen-Macaulay property for $R$ facilitates the computation of the $h$-vector of $R$ by reduction to the artinian case.

For a quadratic (artinian) $K$-algebra  the Koszul property or the existence of a quadratic Gr\"obner basis, or of a Gr\"obner flag   
may sometimes be deduced by inspecting the $h$-vector. 
We collect  some   results on this topic that we will use later on.

\begin{Lemma} 
\label{lemma:tools}
Let $R=\oplus_{i\geq 0}R_i$ be a quadratic standard graded $K$-algebra.
\begin{enumerate} 
\item [(i)] (Conca, Rossi, Valla, \cite[Proposition 2.12]{CRV}) If $H_R(t)=1+nt+t^2$ with $n>1$, then $R$ has a Gr\"obner flag.
\item [(ii)] (Backelin, \cite[Theorem 4.8]{Backelin-thesis}) If $\dim_K R_2 \leq 2$, then  $R$ is Koszul.
\end{enumerate}
\noindent Assume the field $K$ is algebraically closed and of characteristic $\neq 2$.
\begin{enumerate}
\item [(iii)] (Conca, \cite[Theorem 1]{Conca-quadrics})  If $\dim_K R_2 \leq 2$, then $R$ is $G$-quadratic if and only if  
it is not graded isomorphic to the $K$-algebra $K[x,y,z]/(x^2, xy, y^2-xz, yz)$ or its trivial fiber extentions.
\item [(iv)] (Conca, \cite[Theorem 1.1]{Conca-quadrics-codim3}) If $R$ is artinian and $\dim_K R_2=3$, then $R$ is Koszul. Moreover, $R$ is $G$-quadratic
if and only if it is not a trivial fiber extension of $K[x,y,z]/I$, where $I$ is a complete  intersection of three quadrics not containing the square of a linear form.
\end{enumerate}
\end{Lemma}

It is now easier to prove the announced statement for the case when $\embdim(H)$ is less than $5$.

\begin{Theorem}
\label{thm:edim234}
Let $H$ be a   numerical semigroup.
If $\embdim(H) <5$, then $H$ is quadratic if and only if $H$ is  $G$-quadratic.
\end{Theorem}

\begin{proof}
Denote $n=\embdim(H)$ and $R=\gr_\mm K[H]$. Assume $H$ is quadratic.
 If $n=2$  then $R$ is a hypersurface ring, and we are done. If $n=3$, then $e(H) \in \{3, 4\}$ and the result follows from  Theorem \ref{thm:bounds}.

Assume  $n=4$ and then $R$ is Cohen-Macaulay by Proposition \ref{prop:cm-small-n}. 
Denote $\bar{R}=R/(x_1)$. Using Theorem \ref{thm:bounds} we get that $e(H) \in \{4,5,6,8\}$. If $e(H)\neq 5$, then  we are in one of the cases covered by Theorem \ref{thm:bounds}
and the conclusion follows. If $e(H)=5$, then $\ell(\bar{R})=e(H)$ and $\bar{R}$ has the $h$-vector $(1,3,1)$.  By Lemma \ref{lemma:tools}(i) we obtain that $\bar{R}$ has a Gr\"obner flag
which by Lemma \ref{lemma:lifting} produces a Gr\"obner flag for $R$.  Hence $R$ is $G$-quadratic.
\end{proof}

For  a $5$-generated quadratic numerical semigroup there are more possible multiplicities for which the Koszul
 property does not follow easily from Theorem \ref{thm:bounds}.
Our analysis depends on the possible $h$-polynomial of $R=\gr_\mm K[H]$ when $H$ is quadratic.
 We combine the results listed in Lemma \ref{lemma:tools} and Theorem \ref{thm:5-quad-non-cm}
 with computer testing in Singular (\cite{Sing}) for the remaining situations. 
Before giving the main result, we describe our screening strategy performed on the computer.
 
Working under  the assumption that $x_1$ is regular on $R$, 
we analysed the possible quadratic ideals  $J=I^*_H \mod (x_1)$ in $K[x_2, x_3, x_4, x_5]$ generated by monomials and binomials and subject to some restrictions due to our setup. 
Choose $\mathcal{G}$ any minimal generating set for $I_H$ consisting of binomials.
If we denote $\mathcal{G}^*$ the collection of their initial forms, by \cite[Lemma 1.5]{HeS-quad} $\mathcal{G}^*$  generates $I_H^*$ minimally.
Since $x_1$ is regular on $R$, the set $\mathcal{H}$ obtained by letting $x_1=0$ in $\mathcal{G}^*$ is a minimal generating set for $J$ consisting of 
quadratic monomials and possibly of binomials.
 
Since the variables correspond to the generators of $H$ taken in increasing order, there is only a short list of  possible binomials in $\mathcal{H}$: 
\begin{eqnarray*}
p_0=x_2x_5 -x_3x_4, \quad p_1=x_3^2-x_2x_4, \quad p_2=x_3^2-x_2x_5, \\ 
p_3 = x_4^2-x_2 x_5,\quad p_4=x_4^2-x_3x_5. \quad \quad\quad \quad
\end{eqnarray*}
Clearly $p_1$ and $p_2$ can not occur at the same time in $\mathcal{H}$, otherwise $p_1-p_2=x_2(x_4-x_5) \in I_H$, hence $x_4-x_5 \in I_H$, which is false.
Similarly, $p_3$ and $p_4$ may not both occur $\mathcal{H}$. Hence at most $3$ binomials may occur simultaneously in $\mathcal{H}$.

On the other hand, if $p_0$ and $p_1$ occur in $\mathcal{H}$, these occur in $I_H$, too. Hence $a_2+a_5=a_3+a_4$ and $2a_3=a_2+a_4$. Adding these equations we get 
 $a_3+a_5=2a_4$, therefore $p_4 \in I_H$. We get that either  $p_4 \in \mathcal{H}$, or that the monomials in its support are in $\mathcal{H}$.

Arguing similarly we see that if any two of $p_0$, $p_1$ and $p_4$ occur  in $\mathcal{H}$, then the remaining one is in $I_H$ and in $J$.

By Lemma \ref{lemma:quad-semi} we see that $x_2^2, x_5^2 \in  \mathcal{H}$.  Also, if $x_3^2\notin \mathcal{H}$ then either $p_1$ or $p_2$ is in $\mathcal{H}$. 
Similarly, if $x_4^2\notin \mathcal{H}$, then either $p_3$ or $p_4$ is in $\mathcal{H}$.

\begin{Theorem}
\label{thm:5semi-koszul}
Let $H$ be a   numerical semigroup with  $\embdim(H) =5$. 

If the field $K$ is algebraically closed and of characteristic $\neq 2$, then $H$ is Koszul if and only if it is $G$-quadratic.
\end{Theorem}

\begin{proof}
By Theorem \ref{thm:5-quad-non-cm}, if $R=S/I_H^*\cong \gr_\mm K[H]$ is not Cohen-Macaulay then it is $G$-quadratic. 
So it is enough to consider the case when $x_1$ is a nonzero divisor on $R$. 
For $i=2,\dots, 5$, there exist distinct polynomials in $I_H^*$ of the form $f_i= x_i^2-g_i$, where $g_i$ is either zero or a quadratic monomial which is not a pure power. 
Denoting by an overbar $\bar{}$ the image in $\bar{R}=R/(x_1)$ and $<$ the revlex term order induced by $x_5>x_4>\dots$, we have $\ini_<(\bar{f_i})=x_i^2$ for $2\leq i\leq 5$. 
Therefore the Hilbert series of the artinian  graded algebras $\bar{R}$ and $K[x_2,\dots, x_5]/\ini_<(J)$ coincide and  moreover $\ini_<(J)$ contains the squares of all the variables. 
 
In this situation, as noted by Eisenbud, Green and Harris in \cite[Section 4]{EGH-highCastelnuovo}, for any $m$,  if  $h_m=\dim_K \bar{R}_m$ has the binomial decomposition
$$
h_m=  \binom{b_m}{m}+ \binom{b_{m-1}}{m-1}+\dots + \binom{b_1}{1}
$$ 
with $b_m> b_{m-1} > \dots > b_1 \geq 0$, then 
\begin{equation}
\label{eq:growth}
h_{m+1} \leq \binom{b_m}{m+1}+ \binom{b_{m-1}}{m}+\dots + \binom{b_1}{2}.
\end{equation}

The $h$-vector of $\bar{R}$ is $(1, 4, h_2, h_3, h_4)$. By \eqref{eq:growth} we have $0\leq h_2 \leq \binom{4}{2}=6$. 
Our analysis depends on the possible values for $h_2$.
\begin{itemize}
\item
If $h_2=0$, then $\ell(\bar{R})=e(H)=5$, and by Theorem \ref{thm:bounds}, $H$ is $G$-quadratic.

\item
If $h_2=1$, then by \eqref{eq:growth} we get $h_3=0$, hence the $h$-vector of $\bar{R}$ is $(1,4,1)$. 
By Lemma \ref{lemma:tools}(i) and Lemma \ref{lemma:lifting} we get that $\bar{R}$ and $R$ have a Gr\"obner flag, hence they are $G$-quadratic.
\item
If $h_2=2=\binom{2}{2}+ \binom{1}{1}$, from \eqref{eq:growth} we deduce that $h_3=0$. 
  Since $J$ is artinian and the field $K$ is algebraically closed of characteristic $\neq 2$, by Lemma \ref{lemma:tools}(iii) we get that $\bar{R}$, hence also $R$, are $G$-quadratic.

Testing with Singular (\cite{Sing}) the possible candidates for $J$,  it is easy to check that all of them have a quadratic Gr\"obner 
basis with respect to revlex (usually induced by $x_2>x_3>x_4>x_5$). 
All of them possess  a Koszul filtration and in all but one situation presented in Remark \ref{rem:h2-nonflag} there exists a 
Gr\"obner flag with basis $\{x_2, x_3,x_4, x_5\}$.

\item	
If $h_2=3= \binom{3}{2}$, then $h_3 \leq 1$ and $h_4=0$.	Note that $J$ has at least two linearly independent squares of linear forms, namely $x_2^2$ and $x_5^2$.
Under the assumption that $K$ is algebraically closed and of characteristic $\neq 2$, by Lemma \ref{lemma:tools}(iv) 
we infer that $\bar{R}$, hence also $R$, is $G$-quadratic.

Scanning the possible candidates for $J$ by the method described above it turns out that there always exists a Koszul filtration for $\bar{R}$, without any restriction on the field $K$. 
In most cases this filtration is a Gr\"obner flag and the ideal $J$ has a quadratic Gr\"obner basis with respect to revlex (usually induced by $x_2>x_3>x_4>x_5$).
There are though, up to a permutation of the variables, a couple of candidates for $J$ which do not admit a quadratic Gr\"obner basis with respect to any term order.
We present these exceptions in Remark \ref{rem:h3-nonflag}.
\item
If $h_2=4= \binom{3}{2}+ \binom{1}{1}$, then $h_3 \leq 1$ and $h_4=0$.
We scanned the possible candidates for $J$ and we eliminated those ideals where the resolution of $K$ over $\bar{R}$ (computed with Singular \cite{Sing})   
is becoming nonlinear after at most $5$ steps.
All the other candidates had a quadratic Gr\"obner bases with respect to revlex (usually induced by $x_2>x_3>x_4>x_5$) and even a Gr\"obner flag.
All the non-Koszul ideals were among  those with $h_3=0$, hence with $e(H)=9$.
	
\item	
If $h_2=5= \binom{3}{2}+ \binom{2}{1}$, then by \eqref{eq:growth} we get $h_3\leq 2$ and $h_4=0$.
It is well known and easy to see that if $\bar{R}$ is Koszul, then its Poincar\'e  series equals $1/H(-t)$.
It is routine to check that if $h_3=0$ then $1/H(-t)= 1+4t+\cdots -29x^6+\cdots$, and if $h_3=1$ then $1/H(-t)= 1+4t+\cdots -174 x^8+\cdots$. 
Therefore, in either case ${R}$ is not Koszul.
However, if $h_3=2$ then $e(H)= 12=2^4-2^2$ and we may apply Theorem \ref{thm:bounds} to conclude that $H$ is $G$-quadratic.
\item
If $h_2=6$, then $I_H^*$ is a complete intersection, hence  $H$ is $G$-quadratic.
\end{itemize}
This finishes the proof of the theorem.
\end{proof}

\begin{Corollary}
Let $H$ be a numerical semigroup with $\embdim(H)=5$ and $e(H)$ different from $9$, $10$ and $11$. 
 Then  $H$ is quadratic if and only if it is $G$-quadratic.
\end{Corollary}
 
We next present the  situations of Koszul ideals without Gr\"obner flags that occurred when discussing the  cases $h_2=2$ and $h_2=3$ in the proof of Theorem \ref{thm:5semi-koszul}.

\begin{Remark}
\label{rem:h2-nonflag}
{\em  
The   ideal $J_1$ has a quadratic Gr\"obner basis with respect to revlex induced by $x_2>x_3>x_4>x_5$ and the $h$-vector of $\bar{R}=K[x_2,x_3, x_4,x_5]/J_1$ is $(1,4,2)$: 
\begin{equation*}
J_1=(x_2^2, x_5^2, p_1=x_3^2-x_2x_4, x_4^2, x_2x_3, x_2 x_5, x_3 x_4, x_4x_5 ).
\end{equation*} 
The following computations show that 
$$
\mathcal{J}_1= \{ 0, (x_2), (x_2,x_3),  (x_2, x_3, x_5), (x_2, x_4, x_5), (x_2, x_3, x_4, x_5)\}
$$
 is a Koszul filtration for $\bar{R}$:
\begin{eqnarray*}
\quad 0: (x_2)= (x_2, x_3, x_5), \quad (x_2):(x_2,x_5)= (x_2, x_4, x_5), \\
(x_2, x_5): (x_2, x_3, x_5)=(x_2, x_5): (x_2, x_4, x_5)=(x_2, x_3, x_4, x_5),  \\
(x_2, x_3, x_5): (x_2, x_3, x_4, x_5)= (x_2, x_4, x_5): (x_2, x_3, x_4, x_5)=(x_2, x_3, x_4, x_5). 
\end{eqnarray*}
All the computations in these equations are made in $\bar{R}$. We hope there is no risk of confusion.
It is also easy to check that there is no Gr\"obner flag for $\bar{R}$ with basis (the residue classes of) $x_2, x_3, x_4, x_5$.
}
\end{Remark}

\begin{Remark}
\label{rem:h3-nonflag}
{\em
The quotient of $\bar{S}=K[x_2,\dots, x_5]$ modulo either one of the following two ideals  has $h$-vector $(1,4,3)$:
\begin{eqnarray*}
J_2=(x_2^2, x_5^2, p_1=x_3^2-x_2x_4,  p_3= x_4^2-x_2x_5, x_2x_3,  x_3 x_4, x_3 x_5 ), \\
J_3= (x_2^2, x_5^2, p_1=x_3^2-x_2x_4, x_4^2, x_2x_3, x_3x_4, x_3x_5).
\end{eqnarray*}

We claim that none of them has a quadratic Gr\"obner basis with respect to any term order $<$.

Indeed, regarding $J_2$: if $\ini_<(p_1)= x_2x_4$ then the $S$-polynomial $S(p_1, x_2x_3)=x_3^3$. 
Else, in case $\ini_<(p_3)=x_4^2$ we obtain $S(p_1, x_3x_5)= x_2x_4x_5$ 
and in case $\ini_<(p_3)=x_2x_5$ we compute $S(p_1, x_3x_4)=x_2x_4^2$. 
 
Regarding $J_3$: if $\ini_<(p_1)=x_3^2$, then $S(p_1, x_3x_5)=x_2x_4x_5$. Similarly,  if $\ini_<(p_1)=x_2x_4$ then $S(p_1, x_2x_3)= x_3^3$.
It is easy to observe that in any of these cases the computed $S$-polynomial does not reduce to zero 
using the remaining quadrics that generate $J_2$, respectively $J_3$.  Therefore $J_2$ and $J_3$ do not have a quadratic Gr\"obner basis with respect to any term order.

The following computations performed in $\bar{R}=\bar{S}/J_2$, respectively in $\bar{R}=\bar{S}/J_3$,  show that 
\begin{equation*}
\mathcal{J}= \{ 0, (x_5), (x_3,x_5), (x_2, x_5),  (x_2, x_4, x_5), (x_2, x_3, x_5), (x_2, x_3, x_4, x_5)\}
\end{equation*}
 is a Koszul filtration for $\bar{R}$:
\begin{eqnarray*}
\quad 0: (x_5)= (x_3, x_5), (x_5):(x_3,x_5)= (x_2, x_4, x_5), (x_5): (x_2, x_5)=(x_2, x_3, x_5),  \\
  (x_3, x_5):(x_2, x_3, x_5)= (x_2,x_5): (x_2, x_4, x_5)= (x_2, x_3, x_4, x_5), \\ 
	(x_2, x_3, x_5): (x_2, x_3, x_4, x_5)= (x_2, x_4, x_5): (x_2, x_3, x_4, x_5)=(x_2, x_3, x_4, x_5). 
\end{eqnarray*}
}
\end{Remark}
 
\begin{Remark}
{\em
In practice, we were not able to find quadratic numerical semigroups producing the ideals $J_1, J_2, J_3$ in Remarks \ref{rem:h2-nonflag} and \ref{rem:h3-nonflag}. 
If such semigroups do not exist, we could drop the restrictions on the field $K$ in Theorem \ref{thm:5semi-koszul}.
}
\end{Remark}

Based on the proof of Theorem \ref{thm:5semi-koszul} and on the numerical experiments detailed before the proof, 
in Table \ref{table:multiplicities} 
we summarize  with examples our knowledge of the possible $h$-vectors of $\gr_\mm K[H]$, 
grouped by the multiplicity, when $H$ is a $5$-generated quadratic numerical semigroup.
For one of these $h$-vectors we could not find examples of semigroups, 
hence we ask if this list should be further reduced. The abbreviation quad GB indicates that $I_H^*$ has a quadratic Gr\"obner basis.
 
\begin{table}[ht]
\caption{$h$-vectors for quadratic $5$-semigroups}
\centering
\begin{tabular}{r l r r r      } 
\hline \hline 
$e(H)$ & $h$-vector & Remarks on $I_H^*$    &  Example    \\[0.5ex]
\hline 
5 &  $(1,4)$  & quad GB   & $\langle 5,6,7,8,9 \rangle $ \\
\hline
6 &  $(1,4,1)$   & quad GB   & $\langle 6,7,8,9,10 \rangle $  \\
\hline 
7 &  $(1,4,2)$ & quad GB & $\langle 7,8,9,10,11 \rangle $    \\
\hline
8 &  $(1,4,2,1)$  & not CM, quad GB & $\langle  8,12,13,18,35 \rangle $    \\
\phantom{8} &  $(1,4,3)$  & quad GB    & $\langle  8,9,10,11,12 \rangle $      \\
\hline
9 &  $(1,4,3,1)$  & quad GB   &  $\langle  9, 10,11,13,17 \rangle $    \\
\phantom{9} &  $(1,4,4)$  & quad GB  & $\langle  9, 10, 11,12, 15 \rangle $       \\
\phantom{9} &  \phantom{$(1,4,4)$}    & not Koszul  & $\langle  9, 17, 20, 23, 25 \rangle $   \\
\hline
10 &  $(1,4,4,1)$  & quad GB & $\langle 10, 16, 19,22,25 \rangle$      \\
\phantom{10} &  $(1,4,5)$  & not Koszul  & ?      \\
\hline
11 & $(1,4,5,1)$  & not Koszul    & $\langle 11,13,14,15,19 \rangle$  \\
\hline 
12 & $(1,4,5,2)$ &  quad GB  & $\langle 12,14, 16,18,27\rangle$     \\
\hline 
16 & $(1,4,6,4,1)$ & quad GB   & $\langle 16, 17,18,20, 24\rangle$  \\
[1ex]
\hline
\end{tabular}
\label{table:multiplicities}
\end{table} 
 
\begin{Remark}{\em
In recent work (\cite {dAdMM}) D'Anna et al. study the numerical semigroups $H$ for which the   Hilbert function of $\gr_\mm K[H]$ is non-decreasing.
We observe that this is also the case when  $H$ is quadratic and $\embdim(H) \leq 5$.

Indeed, by Proposition \ref{prop:cm-small-n} and the Table \ref{table:multiplicities} we have that the $h$-vector of $\gr_\mm K[H]$ has nonnegative entries, 
hence  the Hilbert function of $\gr_\mm K[H]$ is non-decreasing.
}
\end{Remark}

\medskip

\section{A long proof}
\label{sec:long}

In this section we prove the direct implication of Theorem \ref{thm:5-quad-non-cm}.

Let $H$ be a quadratic numerical semigroup minimally generated by $a_1<\dots <a_5$. Assume 
that $\gr_\mm K[H]$ is not Cohen-Macaulay.

By Theorem \ref{thm:bounds} we see that 
\begin{equation}
4<a_1< 16.
\end{equation}
By Proposition \ref{prop:cm-gr} and Lemma \ref{lemma:quad-semi} the lack of the Cohen-Macaulay property is equivalent to say that there exist $\nu_i \in \{0,1\}$, $i=2,\dots, 5$, 
such that $\sum_{i=2}^5 \nu_i a_i \in a_1+H$ and whenever 
\begin{equation}
\label{eq:sum-cm}
\sum_{i=2}^5 \nu_i a_i= \sum_{i=1}^5 \mu_i a_i, \text{ with integers }\mu_1>0, \mu_2, \dots, \mu_5 \geq 0,
\end{equation}
one has $\sum_{i=2}^5 \nu_i > \sum_{i=1}^5 \mu_i$.

Without loss of generality we may assume that in any equation like \eqref{eq:sum-cm} one has $\nu_i \mu_i=0$ for all $i>1$.
Since $\embdim(H)=5$, at least two of the $\nu_i$'s are positive.
If exactly two of the $\nu_i$'s are equal to $1$, then $\sum_{i=1}^5 \mu_i=1$, $\mu_1=1$, and $a_1 \in \langle a_2, \dots, a_5\rangle$, which is false.
If all $\nu_i$ are positive, then $a_2+\dots +a_5=\mu_1 a_1$ and since $a_1=e(H)$ we get $\mu_1 >4=\sum_{i=2}^5 \nu_i$, 
which contradicts the failure of the Cohen-Macaulay property.

Hence we have to consider only expressions where exactly one $\nu_i$ is zero.
If $\nu_2=0$, then \eqref{eq:sum-cm} is of the form $a_3+a_4+ a_5= a_1+a_2$ or $a_3+a_4+ a_5= 2 a_1$. 
If $\nu_3=0$, then   $a_2+a_4+ a_5= a_1+a_3$ or $a_2+a_4+ a_5= 2 a_1$.
If  $\nu_4=0$, then  $a_2+a_3+ a_5= a_1+a_4$ or $a_2+a_3+ a_5= 2 a_1$.
If $\nu_5=0$, then  $a_2+a_3+ a_4= a_1+a_5$ or $a_2+a_3+ a_4= 2 a_1$.
Since   $a_1<a_2<\dots <a_5$,  the only possibility for \eqref{eq:sum-cm} is
\begin{equation}
\label{eq:sum}
a_2+a_3+a_4=  a_1+a_5.
\end{equation}

By Lemma \ref{lemma:quad-semi}(ii) 
\begin{align}
2a_2  &= u a_1+ v a_3 +  w a_4 +\lambda a_5, \label{eq:a2-normalized} \\ 
2a_3  &= u' a_1+ v' a_2 +  w' a_4 +\lambda' a_5, \label{eq:a3-normalized} \\
2a_4  &= u'' a_1+ v'' a_2 +  w'' a_3 +\lambda'' a_5, \label{eq:a4-normalized}
\end{align}
 for  $u, v, w, \lambda, u', \dots, \lambda''$  nonnegative integers. Moreover, since all $a_i>0$  we may assume, without loss of generality, that
$v,w, \lambda, v',w', \lambda', v'', w'', \lambda'' \in \{0,1\}$. We later refer to these equations as {\em normalized expressions} for $2a_2$, $2a_3$ and $2a_4$, respectively.

We observe that due to the  ordering of the $a_i$'s and to \eqref{eq:sum} we have $a_5>a_3+a_4>2 a_3>2a_2$, hence $\lambda=\lambda'=0$.
Also, \eqref{eq:sum} implies $u''\lambda''=0$, otherwise $a_4 \in \langle a_1, a_2, a_3\rangle$, which is false. Similarly,  $v+w <2$.

The rest of the proof treats the remaining two possibilities: $2 a_2= u a_1 + a_3$, or $2 a_2= u a_1 + a_4$, where we must have $u \geq 1$.
The rather long discussion depends on the coefficients that occur in the normalized expressions \eqref{eq:a3-normalized} and \eqref{eq:a4-normalized}.
We identify six situations  when the tangent cone $\gr_\mm K[H]$ is not Cohen-Macaulay, but, after  reordering, all of them fit into the two families 
$(i)$ and $(ii)$ in  the text of the theorem.
 
\medskip
   
\subsection{\underline{ Case (A)}} Assume 
\begin{equation}
\label{eq:case-a}
2 a_2= u a_1 + a_3 \text{ with }u \geq 1.
\end{equation}

\subsubsection{Case \underline{$w' =0$}.}
			Then  
			\begin{equation}
			2 a_3= u' a_1 + v' a_2 \text{ with } v' \in \{0, 1\}.
			\end{equation}
			 
			If $v'=1$, then $2 a_3= u' a_1+ a_2$, with $u' \geq 1$. Adding to this the equation \eqref{eq:case-a}, after obvious simplifications we obtain 
			$a_2+a_3= (u+u')a_1$, hence 
			$$
			 a_5=(a_2+a_3+a_4)-a_1= (u+u'-1) a_1+a_4 \in \langle a_1, a_4\rangle,
			$$ which is false.
			
			Thus $v'=0$ and $2 a_3= u' a_1$ where $u'>1$ need to be odd and $a_1$ even.
				Together with \eqref{eq:case-a} this gives $2(2a_2-u a_1)= u' a_1$, i.e. 
							$$
							4 a_2= (u'+2u) \cdot a_1.
							$$ 
							Since $u'$ is odd we get $4| a_1$, hence $a_1 \in \{ 8, 12 \}$.
						\begin{enumerate}[leftmargin=0.5cm,label*=\arabic*.]
						\item If \underline{$a_1=8$}, then 
										$$\begin{matrix}   a_2 &=& 2(u'+2u),\\ a_3 &=& 2 a_2- u a_1= 4u' . \end{matrix}$$
										Consider the normalized expression \eqref{eq:a4-normalized}: 
											$$
											2a_4= u'' a_1 + v'' a_2 + w'' a_3 +\lambda'' a_5.
											$$ 
											
										  If $\lambda''=1$, since $a_1, a_2, a_3$ are even, then $a_5$ is even. Using \eqref{eq:sum} we infer that $a_4$ is even, which is false. 
											
										Therefore  $\lambda''=0$ and 
											$$2a_4= u'' a_1 + v'' a_2 + w'' a_3.
											$$ 
											Since $a_2, a_3<a_4$ we should have $u''>0$.
											Because $4|a_1$ and $4|a_3$, we can not have $v''=0$, otherwise $a_4$ is even, and by \eqref{eq:sum} also $a_5$ is even, which gives a contradiction. 
											Hence $v''=1$. We distinguish two situations.
											 \begin{enumerate}[leftmargin=0.5cm, label*=\arabic*.]
												\item If \underline{$w''=1$}, then $2 a_4= u'' a_1 + a_2+ a_3$. This gives 
													\begin{eqnarray}
													\label{eq:sol1}
														a_1 &=& 8, \nonumber  \\ 
														a_3 &=& 4u',  \nonumber \\
														a_2 &=& 4u+2u',  \\
														a_4 &=& 4u''+ 2u+ 3u', \nonumber\\
														a_5 &=& 9u'+6u+4u''-8, \nonumber
													\end{eqnarray}
													which is of the desired form.
												\item If \underline{$w''=0$}, then $2 a_4= u'' a_1 + a_2$. This gives
													\begin{eqnarray}
													\label{eq:sol2}
													a_1 &=& 8, \nonumber  \\ 
													a_3 &=& 4u',  \nonumber \\
													a_2 &=& 4u+2u',  \\
													a_4 &=& 4u''+ 2u+  u', \nonumber\\
													a_5 &=& 7u'+6u+4u''-8, \nonumber
												\end{eqnarray}
											   which is of the desired form.
											 \end{enumerate}
						\item If \underline{$a_1=12$}, then
							\begin{eqnarray*}
								a_2 &=& 3u'+ 6u,\\ 
								a_3 &=& 6u'.
							\end{eqnarray*}
							In the normalized expression 
											$$
											\quad 2a_4= u'' a_1 + v'' a_2 + w'' a_3 +\lambda'' a_5 
											$$ 
							we add $a_4+(1-v'') a_2+ (1-w'')a_3$ to both sides and using \eqref{eq:sum} we get
							$$
							3 a_4 +(1-v'')a_2+(1-w'')a_3= (1+u'')a_1+ (1+\lambda'')a_5,
							$$
							hence $3|a_5$, and by \eqref{eq:sum} also $3|a_4$, hence $\gcd(a_1, \dots, a_5) >1$, a contradiction.
						\end{enumerate}
\subsubsection{Case \underline{$w'=1$}.}
\label{ssec-test}
Then 
\begin{equation*}
2a_3= u' a_1+ v' a_2+a_4 \text{ with } v'\in \{0,1\}.
\end{equation*}
If $v'=1$, then $2a_1+ 2a_5= 2a_2+ 2a_3+ 2 a_4= (u a_1+a_3)+ (u' a_1+ a_2+a_4)+ 2a_4= (u'+u) a_1+ (a_2+a_3+a_4)+ 2a_4$. 
This gives $a_5=(u'+u-1)a_1+2 a_4$, which is false.

\noindent Therefore $v'=0$ and
\begin{equation}
\label{eq:a3}
2a_3= u' a_1+  a_4 \text{ with } u' >0.
\end{equation} 
Consider the normalized expression \eqref{eq:a4-normalized}
\begin{equation*}
									2a_4= u'' a_1 + v'' a_2 + w'' a_3 +\lambda'' a_5.
\end{equation*} 
If $\lambda'' =1$,  since $u''\lambda''=0$, we get $u''=0$. Equation  \eqref{eq:a4-normalized} becomes
\begin{equation*}
2 a_4= v'' a_2+ w'' a_3+ a_5.
\end{equation*}
To this we add \eqref{eq:a3} and  $2a_2= u a_1+a_3$ from \eqref{eq:case-a} and after using \eqref{eq:sum} we get that
$$
 a_5= (u'+u-2) a_1 + v'' a_2+ (w''+1) a_3+ a_4 \in \langle a_1, a_2, a_3, a_4\rangle,
$$
which is false.

\noindent Thus $\lambda''=0$ and 
$$
2 a_4= u'' a_1+ v'' a_2+ w'' a_3.
$$
To this we   add \eqref{eq:a3} and $2a_2= u a_1+a_3$ and, after using \eqref{eq:sum}, we see that
$$
2 a_5=(u''+u'+u-2) a_1+  v'' a_2+ (w''+1)a_3+ a_4.
$$
If $v''>0$, by \eqref{eq:sum} we get $a_5\in \langle a_1, a_2, a_3, a_4 \rangle$, which is false.

\noindent Hence $v''=0$ and \eqref{eq:a4-normalized} becomes
$$
2a_4= u'' a_1+ w'' a_3= u'' a_1+ w''(2a_2-u a_1)= (u''-w''u)a_1+ 2w'' a_2 \text{ with } w''\in\{0, 1\}.
$$
From \eqref{eq:a3} we extract
$$
a_4= 2a_3-u' a_1= 2(2a_2-u a_1)- u' a_1=4a_2-(2u+u')a_1 
$$ 
which we replace in the previous equation for $2a_4$. Routine manipulation gives
\begin{equation}
\label{eq:combo}
(8-2w'')a_2=((4-w'')u+2u'+u'') a_1.
\end{equation}
\begin{enumerate}[leftmargin=0.5cm,label*=\arabic*.]
\item If \underline{$w''=0$}, the equations \eqref{eq:combo}, \eqref{eq:case-a} and \eqref{eq:a4-normalized} together with \eqref{eq:sum} yield
	\begin{eqnarray*}
	a_2 &=& \frac{4u+2u'+u''}{8} \cdot a_1,  \\
	a_3 &=& \frac{2u'+u''}{4} \cdot a_1,   \\
	a_4 &=& \frac{u''}{2} \cdot a_1,\\
	a_5 &=& \frac{4u+6u'+7u''-8}{8} \cdot a_1.  
	\end{eqnarray*}
	From here we infer that $u''$ is odd (otherwise $a_1$ divides $a_4$, which is false) and consequently $8|a_1$. 
	Hence $a_1=8$ and the generators of the semigroup are  
	\begin{eqnarray}
	\label{eq:sol3}
	a_1 &=& 8, \nonumber \\
	a_4 &=& 4u'',  \nonumber \\
	a_3 &=& 2u''+4u',   \\
	a_2 &=& u''+ 2u'+ 4u,  \nonumber \\
	a_5 &=& 4u+6u'+ 7u''-8,  \nonumber
	\end{eqnarray}
	as desired.
\item If \underline{$w''=1$}, equation \eqref{eq:combo} together with \eqref{eq:case-a},  \eqref{eq:a4-normalized} and \eqref{eq:sum} give after routine computations 
\begin{eqnarray*}
	a_2 &=& \frac{3u+2u'+u''}{6} \cdot a_1,  \\
	a_3 &=& \frac{2u'+u''}{3} \cdot a_1,   \\
	a_4 &=& \frac{u'+2u''}{3} \cdot a_1,\\
	a_5 &=& \frac{3u+8u'+7u''-6}{6} \cdot a_1.  
	\end{eqnarray*}
We note that $a_5= a_2+ (u'+u''-1) a_1 \in \langle a_1, a_2 \rangle$, which is false.	
\end{enumerate}

\medskip

\subsection{\underline{Case (B)}}
 Assume 
\begin{equation}
\label{eq:case-b}
2 a_2= u a_1 + a_4, \text{  with }u \geq 1.
\end{equation}
 
If in the normalized expression 
$$
2 a_4= u'' a_1+ v'' a_2+ w'' a_3+\lambda'' a_5 \text{\quad with } v'', w'', \lambda'' \in\{0,1\}
$$ 
we had $\lambda''=1$, then  $u''=0$.
Combined with \eqref{eq:case-b}, we get
$$
2a_2+a_4= ua_1+ 2 a_4= u a_1 +v'' a_2 +w''a_3+a_5.
$$
Using \eqref{eq:sum} and  the latter equation we obtain
$$
(1-v'')a_2= (u-1)a_1+ (w''+1) a_3 >0,
$$
which is a contradiction since $v'' \leq 1$. Therefore $\lambda''=0$.

\subsubsection{Case \underline{$w'=v'=0$}.} 
  \label{ss:test} 
	Then $u'>0$ and \eqref{eq:a3-normalized} becomes
	\begin{equation}
		2 a_3= u' a_1.
	\end{equation}
 From this and \eqref{eq:case-b} we get $a_3=u'a_1/2$, $a_4=2a_2-u a_1$ and that $a_1$ is even and $u'$ is odd.
 
We substitute in the normalized expression above  the values for $a_3$ and $a_4$ in terms of $a_1$ and $a_2$ and we get
$$
(4-v'') \cdot a_2= \left(2u+u''+ w''\frac{u'}{2}\right) \cdot a_1.
$$
	\begin{enumerate}[leftmargin=0.5cm,label*=\arabic*.]
		\item If \underline{$v''=1$}, the previous equation becomes
		$$
		3a_2= \left(2u+u''+w''\frac{u'}{2}\right) \cdot a_1. 
		$$
		We consider the possible even values of $a_1$. 
		\begin{enumerate}[leftmargin=0.5cm, label*=\arabic*.]
		 \item If \underline{$a_1=6$} we obtain the other generators
			\begin{eqnarray*}
				a_2 &=& 4u+2u''+w''u',\\
				a_3 &=& 3u', \\
				a_4 &=&2u+4u''+2w''u', \\
				a_5 &=& 6u+6u''+ (3w''+3)u'-6.
			\end{eqnarray*}
			Note that $a_5 = (u+u''-1) a_1+ (w''+1) a_3\in \langle a_1, a_3 \rangle$, which is false.
		 \item If \underline{$a_1=8$ or $12$}, then $a_2,a_3, a_4$ are even, hence $a_5$ is even as well, a contradiction.
		 \item If \underline{$a_1=10$}, then it easy to see that $a_2,a_3, a_4, a_5$ are divisible by $5$, which is false.
		 \item If \underline{$a_1=14$}, then all the generators are divisible by $7$, which is false.
		\end{enumerate}
		\item If \underline{$v''=0$}, then
			$$
				4a_2= \left(2u+u''+w''\frac{u'}{2}\right)\cdot a_1,
			$$
			which forces $a_1$ to be even.
			\begin{enumerate}[leftmargin=0.5cm, label*=\arabic*.]
				\item If \underline{$w''=0$} we get $2 a_4= u'' a_1$ and $4 a_2= (2u+u'') a_1$. Therefore $u''$ is odd, $a_1$ is divisible by $4$, hence $a_1\in \{4,8\}$,
				and the other generators are
					\begin{eqnarray*}
						a_2 &=& \frac{2u+u''}{4}\cdot a_1,\\
						a_3 &=& \frac{u'}{2}\cdot a_1, \\
						a_5 &=& \frac{3u''+2u+2u'-4}{4}\cdot a_1.
					\end{eqnarray*}
					It is immediate to note that if $a_1=8$ all generators are even, while if $a_1=12$ all of them are divisible by $3$. 
					None of these situations may hold. 
				\item If \underline{$w''=1$} we obtain 
					$$
						8 a_2= (4u+2u''+u') \cdot a_1.
					$$
					Since $u'$ is odd we get $a_1=8$. From the other equations we compute the   other generators 
					 \begin{eqnarray}
						\label{eq:sol4}
							a_3 &=& 4u', \nonumber\\
							a_4 &=& 4u''+2u', \nonumber \\
							a_2 &=& 4u+ 2u''+u', \\
							a_5 &=& 4u+6u''+7u'-8,  \nonumber
					 \end{eqnarray}
					which turn out to be of the desired form.
			\end{enumerate}
	\end{enumerate}
\subsubsection{Case \underline{$w'= 0$ and $v'=1$}.} 
Then $2 a_3= u'a_1+a_2$ with $u'>0$. 

Using \eqref{eq:case-b} and the normalized equation \eqref{eq:a4-normalized} we get 
$$
  2a_1+ 2a_5= 2 a_2+ 2a_3+ 2a_4=(u'+u''+u)a_1+ (v''+1) a_2+ w'' a_3+ a_4,
$$ 
hence $w''=0$. 
		\begin{enumerate}[leftmargin=0.5cm,label*=\arabic*.]
			\item If \underline{$v''=1$}, by adding the equations $2a_2=u a_1+ a_4$ and $2a_4= u'' a_1+ a_2$ we get $a_2+a_4=(u+u'')a_1$.
				Therefore $a_5=(a_2+a_4)+ a_3-a_1= (u+u''-1) a_1+a_3$, which is false.
			\item If \underline{$v''=0$}, then $2a_4=u'' a_1$ with $u''$ odd and $a_1$ even. 
			Since $2a_2=u a_1+ a_4$ we get   $4 a_2= (2u+ u'')\cdot a_1$. This   implies that $4|a_1$, hence $a_1\in \{8,12\}$. 
			
			If \underline{$a_1=12$}, then $a_2=3(2u+u'')$. Since  $2 a_3=u' a_1+ a_2$ and $a_4= ua_1-2a_2$ we derive that $a_3$ and $a_4$ are divisible by $3$, 
			   hence also $3|a_5$, which is false.
				
			If \underline{$a_1=8$}, the rest of the generators are 
				\begin{eqnarray}
					\label{eq:sol5}
					a_4 &=& 4u'', \nonumber \\
					a_2 &=& 4u+2u'', \nonumber \\
					a_3 &=& 2u+u''+4u', \\
					a_5 &=& 6u+7u''+4u'-8, \nonumber
				\end{eqnarray}
			which are of the desired format.	
		\end{enumerate}
\subsubsection{Case \underline{$w'=1$}.} 
  Then
   \begin{equation}
			\label{eq:new-new-a3}
			2a_3=u' a_1+ v' a_2+a_4.
	 \end{equation}
  Since  in the normalized expression \eqref{eq:a4-normalized} we have $\lambda''=0$ and $v'',w'' \leq 1$, then $u'' >0$.
	 	\noindent By \eqref{eq:case-b} and \eqref{eq:new-new-a3} we may write
			\begin{eqnarray*}
				2 a_5 &=& 2a_2+ 2a_3+2a_4-2a_1 \\
				      &=& (u a_1+a_4) + (u'a_1+v' a_2 +a_4) + (u''a_1+v''a_2+w'' a_3)-2a_1 \\
							&=& (u+u'+u''-2) a_1+ (v'+v'')a_2 +w'' a_3 +2a_4.
			\end{eqnarray*}
		 \begin{enumerate}[leftmargin=0.5cm,label*=\arabic*.]
			\item If \underline{$w''=1$ and $v'+v'' >0$}, we get $a_5 \in \langle a_1, a_2, a_3, a_4 \rangle$, which is false.
			\item If \underline{$w''=1$ and $v'=v''=0$}, summing the equations
			    \begin{eqnarray*}
							2a_4 &=& u'' a_1+ a_3 \\
							2 a_3 &=& u' a_1 +a_4
					\end{eqnarray*}
					we obtain that $a_3+a_4=(u'+u'')a_1$, which, together with \eqref{eq:sum} yields  $a_5=(u'+u''-1)a_1+a_2 \in \langle a_1, a_2 \rangle$, a contradiction.
			\item If \underline{$w''=0$}, then $2a_4= u'' a_1+ v'' a_2$, and after substituting in here $a_4= 2a_2-ua_1$ (from \eqref{eq:case-b}) we get
			$$
			(4-v'')a_2= (2u+u'')a_1.
			$$
			\begin{enumerate}[leftmargin=0.5cm, label*=\arabic*.]
			\item			If \underline{$v''=0$}, then $a_4= u'' a_1/2$ and   $u''$ is odd.
			 Other generators are obtained immediately:
			   \begin{eqnarray*}
					 a_2 &=& \frac{2u+u''}{4} \cdot a_1, \\
					 a_3 &=& \frac{4u'+2u''+v'(2u+u'')}{8} \cdot a_1. \\
				 \end{eqnarray*}
			Since $u''$ is odd  we deduce that $4|a_1$, hence $a_1 \in \{8,12 \}$.
			
			If \underline{$a_1=12$}, because the denominators of $a_2,a_3$ and $a_4$ are powers of $2$ we see that $a_2, a_3, a_4$ are divisible by $3$, and the same holds for $a_5$, a  contradiction. \\
			\noindent Assume \underline{$a_1=8$}.  If $v'=0$ then we note that $a_2, a_3, a_4$ are even, hence $a_5$ is even, too. This is false. Hence $v'=1$, and the generators of the semigroup are
			\begin{eqnarray}
				\label{eq:sol6}
				a_1 &=& 8, \nonumber\\
				a_4 &=& 4u'', \nonumber \\
				a_2 &=& 4u+2u'', \\
				a_3 &=& 4u'+ 2u+ 3u'', \nonumber \\
				a_5 &=& 4u'+6u+ 9u''-8, \nonumber
			\end{eqnarray}
			which is on our list.
			\item If \underline{$v''=1$}, then $3a_2=(2u+u'') a_1$ which implies  $3 \nmid 2u+u''$ and  $3|a_1$.
			We get 
			\begin{eqnarray*}
				a_2 &=& \frac{2u+u''}{3} \cdot a_1, \\
				a_4 &=& 2a_2-u= \frac{2u''+u}{3} \cdot a_1.
			\end{eqnarray*}
			By \eqref{eq:sum} we see that  $a_5=a_2+a_3+a_4-a_1= (u+u''-1) a_1+ a_3$, which is false.
			\end{enumerate}
		 \end{enumerate}
The proof of Theorem \ref{thm:5-quad-non-cm} is now complete. $\square$

\medskip

{\bf Acknowledgement}.
We gratefully acknowledge the use of  the computer algebra system Singular \cite{Sing} for our numerical experiments. 
 We thank Aldo Conca for pointing our attention to \cite{Conca-quadrics-codim3}. We thank Mihai Cipu for verifying the proof of Theorem \ref{thm:5-quad-non-cm} 
and the suggested improvements. 
We   also thank  Francesco Strazzanti who communicated us the  quadratic semigroup with $h$-vector $(1,4,3,1)$ from Table \ref{table:multiplicities}.

The   author was  supported by  a grant of the Romanian Ministry of Education, CNCS--UEFISCDI, project number PN-II-RU-PD-2012-3--0656.

{}

\end{document}